\documentclass[12pt]{article}

\usepackage{amsmath,amsthm,amssymb}
\usepackage[top=30truemm,bottom=30truemm,left=25truemm,right=25truemm]{geometry}
\usepackage[dvips]{graphicx,color,psfrag}

\theoremstyle{plain}
\newtheorem{definition}{Definition}[section]
\newtheorem{thm}[definition]{Theorem}

\newtheorem{lem}[definition]{Lemma}

\newtheorem{rem}[definition]{Remark}

\let\Re\relax
\DeclareMathOperator{\Re}{Re}
\let\Im\relax
\DeclareMathOperator{\Im}{Im}

\title{Integer circulant determinants of order $16$}
\author{Yuka Yamaguchi and Naoya Yamaguchi}

\begin{document}

\maketitle

\begin{abstract}
We solve Olga Taussky-Todd's circulant problem in the case of order $16$. 
\end{abstract}

\section{Introduction}
For a positive integer $n$, 
the determinant 
$$
D(a_{0}, a_{1}, \ldots, a_{n - 1}) := \det{\begin{pmatrix} a_{0} & a_{n - 1} & \cdots & a_{1} \\ a_{1} & a_{0} & \cdots & a_{2} \\ \vdots & \vdots & \ddots & \vdots \\ a_{n - 1} & a_{n - 2} & \cdots & a_{0} \end{pmatrix}}
$$
is called a circulant determinant of order $n$. 
When all entries $a_{i}$ are integers, 
the circulant determinant is called an integer circulant determinant. 
Let $S(n)$ denote the set of all possible values of integer circulant determinants of order $n$: 
$$
S(n) := \left\{ D(a_{0}, a_{1}, \ldots, a_{n - 1}) \: \mid \: a_{i} \in \mathbb{Z} \right\}. 
$$
The following is well-known: 
$$
S(2) = \left\{ a_{0}^{2} - a_{1}^{2} \mid a_{0}, a_{1} \in \mathbb{Z} \right\} = \mathbb{Z}_{{\rm odd}} \cup 4 \mathbb{Z}, 
$$
where $\mathbb{Z}_{{\rm odd}}$ denotes the set of all odd numbers. 
In 1977, at the meeting of the American Mathematical Society, Olga Taussky-Todd proposed a problem that is to determine $S(n)$ for any $n$. 
We call the problem Olga Taussky-Todd's circulant problem. 
A few years later, 
in the case of $n = p$, $2 p$ and $p^{2}$ with $p$ is an odd prime, 
the problem was solved \cite{MR624127}, \cite{MR601702}, \cite{MR550657}. 
Further, in 2012, Kaiblinger \cite[Theorem~1.1]{MR2914452} presented the following: 
\begin{enumerate}
\item[\scriptsize$\bullet$] $S(4) = \mathbb{Z}_{{\rm odd}} \cup 16 \mathbb{Z}$; 
\item[\scriptsize$\bullet$] $S(8) = \mathbb{Z}_{{\rm odd}} \cup 32 \mathbb{Z}$; 
\item[\scriptsize$\bullet$] $\mathbb{Z}_{{\rm odd}} \cup 2^{2 k - 1} \mathbb{Z} \subset S(2^{k}) \subset \mathbb{Z}_{{\rm odd}} \cup 2^{k + 2} \mathbb{Z}, \quad k \geq 4$. 
\end{enumerate}
These results tempt to  ask a question whether $S(16) = \mathbb{Z}_{{\rm odd}} \cup 64 \mathbb{Z}$. 
This paper gives a negative answer to the question as follows. 

\begin{thm}\label{thm:1.1}
The following holds: 
\begin{align*}
S(16) = \mathbb{Z}_{{\rm odd}} &\cup 128 \mathbb{Z} \\ 
&\cup \left\{ 64 p m \mid p = a^{2} + b^{2} \equiv 1, \: a + b \equiv 3, 5 \: \: ({\rm mod \: 8}), \: m \in \mathbb{Z} \right\} \\ 
&\cup \left\{ 64 p m \mid p \equiv 5 \: \: ({\rm mod \: 8}), \: m \in \mathbb{Z} \right\} \\ 
&\cup \left\{ 64 p^{2} m \mid p \equiv 3 \: \: ({\rm mod \: 8}), \: m \in \mathbb{Z} \right\}, 
\end{align*}
where $p$ denotes a prime. 
\end{thm}

The group determinant can be regarded as a generalization of the circulant determinant, 
where, for a finite group $G$, 
assigning an indeterminate $x_{g}$ for each $g \in G$, 
the group determinant was defined as $\det{\left( x_{g h^{- 1}} \right)}_{g, h \in G}$. 
When $G$ is the cyclic group of order $n$, 
the group determinant becomes the circulant determinant of order $n$. 
Thus Olga Taussky-Todd's circulant problem is naturally generalized to the problem that is to determine the set 
$$
S(G) := \left\{ \det{\left( x_{g h^{- 1}} \right)}_{g, h \in G} \mid x_{g} \in \mathbb{Z} \right\}
$$
for any finite group $G$. 
Recently, several results were obtained \cite{MR4363104}, \cite{MR4313424}, \cite{MR3998922}, \cite{MR4056860}. 
In particular, as mentioned in \cite{MR4363104}, 
for every group $G$ of order at most $15$, 
$S(G)$ was determined. 
Theorem~$\ref{thm:1.1}$ determines $S(G)$ in the case that $G$ is the cyclic group of order $16$, 
which is one of the smallest unsolved group. 

\section{Preliminaries}

Let $\zeta_{n}$ be a primitive $n$-th root of unity. 
It is well-known that 
\begin{align*}
D(a_{0}, a_{1}, \ldots, a_{15}) = \prod_{l = 0}^{15} f(\zeta_{16}^{l}), \quad f(x) := \sum_{k = 0}^{15} a_{k} x^{k} \in \mathbb{Z}[x]. 
\end{align*}
In the same way as in \cite{MR4363104}, 
we factorize the right-hand side as a product of integer norms 
\begin{align*}
D(a_{0}, a_{1}, \ldots, a_{15}) = \prod_{d \mid 16} N_{d}, \quad N_{d} := \prod_{\gcd(l, 16) = d} f(\zeta_{16}^{l}) \in \mathbb{Q} \cap \mathbb{Z}[\zeta_{16}] = \mathbb{Z}. 
\end{align*}
Let 
\begin{align*}
b_{k} &:= (a_{k} + a_{k + 8}) + (a_{k + 4} + a_{k + 12}), && 0 \leq k \leq 3, \\ 
c_{k} &:= (a_{k} + a_{k + 8}) - (a_{k + 4} + a_{k + 12}), && 0 \leq k \leq 3, \\ 
\hspace{3cm} d_{k} &:= (a_{k} - a_{k + 8}) + \sqrt{- 1} (a_{k + 4} - a_{k + 12}), && 0 \leq k \leq 3, \hspace{3cm}\\ 
e_{k} &:= a_{k} - a_{k + 8}, && 0 \leq k \leq 7.  
\end{align*}

\begin{rem}\label{rem:2.1}
The following three conditions are equivalent: 
\begin{enumerate}
\item[$(1)$] $b_{0} \not\equiv b_{2}, \: b_{1} \not\equiv b_{3} \pmod{2}$; 
\item[$(2)$] $c_{0} \not\equiv c_{2}, \: c_{1} \not\equiv c_{3} \pmod{2}$; 
\item[$(3)$] $e_{0} + e_{4} \not\equiv e_{2} + e_{6}, \: e_{1} + e_{5} \not\equiv e_{3} + e_{7} \pmod{2}$. 
\end{enumerate}
\end{rem}

Let $\alpha_{1} := f(\zeta_{16}) f(\zeta_{16}^{5}) f(\zeta_{16}^{9}) f(\zeta_{16}^{13})$, $\alpha_{2} := f(\zeta_{8}) f(\zeta_{8}^{5})$. 
By direct calculation, 
we have the following two lemmas. 

\begin{lem}\label{lem:2.3}
The following hold: 
\begin{align*}
\Re{(\alpha_{1})} 
&= e_{0}^{4} + e_{4}^{4} - e_{2}^{4} - e_{6}^{4} - 6 e_{0}^{2} e_{4}^{2} + 6 e_{2}^{2} e_{6}^{2} - 2 (e_{1}^{2} - e_{5}^{2}) (e_{3}^{2} - e_{7}^{2}) \\ 
&\quad + 4 (e_{2} e_{6} + e_{1} e_{7} + e_{3} e_{5}) (e_{0}^{2} - e_{4}^{2}) - 4 (e_{0} e_{6} + e_{2} e_{4} - e_{1} e_{5}) (e_{1}^{2} - e_{5}^{2}) \\ 
&\quad + 4 (e_{0} e_{4} + e_{1} e_{3} - e_{5} e_{7}) (e_{2}^{2} - e_{6}^{2}) - 4 (e_{0} e_{2} - e_{4} e_{6} + e_{3} e_{7} ) (e_{3}^{2} - e_{7}^{2}) \\ 
&\quad - 8 e_{0} e_{2} e_{1} e_{5} + 8 e_{0} e_{4} e_{1} e_{3} - 8 e_{0} e_{4} e_{5} e_{7} + 8 e_{0} e_{6} e_{3} e_{7} + 8 e_{2} e_{4} e_{3} e_{7} \\ 
&\quad - 8 e_{2} e_{6} e_{1} e_{7} - 8 e_{2} e_{6} e_{3} e_{5} + 8 e_{4} e_{6} e_{1} e_{5} + 8 e_{1} e_{3} e_{5} e_{7}, \\ 
\Im{(\alpha_{1})} 
&= e_{3}^{4} + e_{7}^{4} - e_{1}^{4} - e_{5}^{4} - 6 e_{3}^{2} e_{7}^{2} + 6 e_{1}^{2} e_{5}^{2} - 2 (e_{0}^{2} - e_{4}^{2}) (e_{2}^{2} - e_{6}^{2}) \\ 
&\quad + 4 (e_{0} e_{4} - e_{1} e_{3} + e_{5} e_{7}) (e_{0}^{2} - e_{4}^{2}) + 4 (e_{0} e_{2} - e_{4} e_{6} - e_{3} e_{7} ) (e_{1}^{2} - e_{5}^{2}) \\ 
&\quad - 4 (e_{2} e_{6} - e_{1} e_{7} - e_{3} e_{5}) (e_{2}^{2} - e_{6}^{2}) - 4 (e_{0} e_{6} + e_{2} e_{4} + e_{1} e_{5}) (e_{3}^{2} - e_{7}^{2}) \\ 
&\quad + 8 e_{0} e_{2} e_{4} e_{6} - 8 e_{0} e_{2} e_{3} e_{7} + 8 e_{0} e_{4} e_{1} e_{7} + 8 e_{0} e_{4} e_{3} e_{5} - 8 e_{0} e_{6} e_{1} e_{5} \\ 
&\quad - 8 e_{2} e_{4} e_{1} e_{5} + 8 e_{2} e_{6} e_{1} e_{3} - 8 e_{2} e_{6} e_{5} e_{7} + 8 e_{4} e_{6} e_{3} e_{7}, \\ 
\Re{(\alpha_{2})}  &= c_{0}^{2} - c_{2}^{2} + 2 c_{1} c_{3}, \\ 
\Im{(\alpha_{2})}  &= c_{3}^{2} - c_{1}^{2} + 2 c_{0} c_{2}. 
\end{align*}
\end{lem}

\begin{lem}\label{lem:2.2}
The following hold: 
\begin{enumerate}
\item[$(1)$] $N_{1} = \alpha_{1} \overline{\alpha_{1}}$; 
\item[$(2)$] $N_{2} = \alpha_{2} \overline{\alpha_{2}} = (c_{0}^{2} - c_{2}^{2} + 2 c_{1} c_{3})^{2} + (c_{3}^{2} - c_{1}^{2} + 2 c_{0} c_{2})^{2}$; 
\item[$(3)$] $N_{4} = f(\sqrt{-1}) f(- \sqrt{-1}) = (b_{0} - b_{2})^{2} + (b_{1} - b_{3})^{2}$; 
\item[$(4)$] $N_{8} N_{16} = f(1) f(- 1) = (b_{0} + b_{2})^{2} - (b_{1} + b_{3})^{2}$, 
\end{enumerate}
where $\overline{\alpha}$ denotes the complex conjugate of $\alpha$. 
\end{lem}
Lemma~$\ref{lem:2.3}$ implies $\alpha_{1}, \alpha_{2} \in \mathbb{Z}[\sqrt{- 1}]$. 
Thus, from Lemma~$\ref{lem:2.2}$~$(1)$--$(3)$, 
we find that $N_{1}$, $N_{2}$ and $N_{4}$ are expressible in the form $a^{2} + b^{2}$. 
It is well-known that a positive integer $n$ is expressible as a sum of two squares if and only if in the prime factorization of $n$, 
every prime of the form $4 m + 3$ occurs an even number of times. 
From this fact, we have the following lemma. 

\begin{lem}\label{lem:2.5}
For any prime $p$ of the form $4 m + 3$ and $d \in \{ 1, 2, 4 \}$, 
it holds that if $p^{2 k - 1} \mid N_{d}$ with $k \geq 1$, 
then $p^{2 k} \mid N_{d}$. 
\end{lem}

\section{Conditions for integer norms}

In this section, 
we prove the following theorem used to determine impossible values in the next section. 

\begin{thm}\label{thm:3.1}
The following three conditions are equivalent: 
\begin{enumerate}
\item[$(1)$] $D(a_{0}, a_{1}, \ldots, a_{15}) = \prod_{d \mid 16} N_{d} \in 64 \mathbb{Z}_{{\rm odd}}$; 
\item[$(2)$] $N_{1}, N_{2} \in 2 \mathbb{Z}_{{\rm odd}}$ and $N_{4} N_{8} N_{16} \in 16 \mathbb{Z}_{{\rm odd}}$; 
\item[$(3)$] $N_{1}, N_{2}, N_{4} \in 2 \mathbb{Z}_{{\rm odd}}$ and $N_{8} N_{16} \in 8 \mathbb{Z}_{{\rm odd}}$. 
\end{enumerate}
\end{thm}

Theorem~$\ref{thm:3.1}$ immediately follows from the following three lemmas. 

\begin{lem}\label{lem:3.2}
We have $N_{1} \equiv N_{2} \equiv N_{4} \equiv N_{8} \equiv N_{16} \pmod{2}$. 
\end{lem}
\begin{proof}
Since $f(\zeta_{16}^{k}) \equiv f(1) \pmod{1 - \zeta_{16}}$ holds for any $k \in \mathbb{Z}$ and $m^{2} \equiv m \pmod{1 - \zeta_{16}}$ holds for any $m \in \mathbb{Z}$, 
we have 
$$
N_{1} \equiv N_{2} \equiv N_{4} \equiv N_{8} \equiv N_{16} \pmod{1 - \zeta_{16}}. 
$$
Therefore, 
from $(1 - \zeta_{16}) \mathbb{Z}[\zeta_{16}] \cap \mathbb{Z} = 2 \mathbb{Z}$, 
the lemma is proved. 
\end{proof}

\begin{lem}\label{lem:3.3}
We have 
$N_{4} N_{8} N_{16} \in \mathbb{Z}_{{\rm odd}} \cup 16 \mathbb{Z}$. 
\end{lem}
\begin{proof}
Using the result for $S(4)$ obtained by Kaiblinger (see the introduction), 
we have 
$$
N_{4} N_{8} N_{16} = \prod_{l = 0}^{3} f(\zeta_{4}^{l}) = \prod_{l = 0}^{3} (b_{0} + b_{1} \zeta_{4}^{l} + b_{2} \zeta_{4}^{2 l} + b_{3} \zeta_{4}^{3 l} ) \in S(4) = \mathbb{Z}_{{\rm odd}} \cup 16 \mathbb{Z}. 
$$
\end{proof}

\begin{lem}\label{lem:3.4}
The following holds: $N_{2} \in 2 \mathbb{Z}_{{\rm odd}} \: \:  \text{if and only if} \: \: N_{4} \in 2 \mathbb{Z}_{{\rm odd}}$. 
\end{lem}
\begin{proof}
From Lemma~$\ref{lem:2.2}$~$(2)$ and $(3)$, 
we have 
\begin{align*}
N_{2} \in 2 \mathbb{Z}_{{\rm odd}} 
\iff c_{0}^{2} - c_{2}^{2} \equiv c_{3}^{2} - c_{1}^{2} \equiv 1 \pmod{2} 
\iff c_{0} \not\equiv c_{2}, \: c_{1} \not\equiv c_{3} \pmod{2}, \\ 
N_{4} \in 2 \mathbb{Z}_{{\rm odd}} \iff b_{0} - b_{2} \equiv b_{1} - b_{3} \equiv 1 \pmod{2} \iff b_{0} \not\equiv b_{2}, \: b_{1} \not\equiv b_{3} \pmod{2}, 
\end{align*}
respectively. 
Therefore, from Remark~$\ref{rem:2.1}$, 
the lemma is proved. 
\end{proof}

\section{Impossible values}\label{Section4}

In this section, 
we present impossible values. 
We will see, in the last section, that $S(16)$ includes every integer that is not mentioned in the following theorem. 

\begin{thm}\label{thm:4.1}
Let $p_{i} = a_{i}^{2} + b_{i}^{2} \equiv 1 \pmod{8}$ be a prime 
with $a_{i} \pm b_{i} \in \{8 m \pm 1 \mid m \in \mathbb{Z} \}$ for each $1 \leq i \leq r$,  
let $p_{r + 1}, \ldots, p_{r + s} \equiv 7 \pmod{8}$ be primes, 
let $q_{1}, \ldots, q_{t} \equiv 3 \pmod{8}$ be distinct primes, 
and let $k_{1}, \ldots, k_{r + s}$ be non-negative integers. 
Then 
\begin{align*}
64 p_{1}^{k_{1}} \cdots p_{r}^{k_{r}} p_{r + 1}^{k_{r + 1}} \cdots p_{r + s}^{k_{r + s}}, \quad 
64 p_{1}^{k_{1}} \cdots p_{r}^{k_{r}} p_{r + 1}^{k_{r + 1}} \cdots p_{r + s}^{k_{r + s}} q_{1} \cdots q_{t} 
\notin S (16). 
\end{align*}
\end{thm}

To prove Theorem~$\ref{thm:4.1}$, 
we use the following six lemmas.

\begin{lem}\label{lem:4.2}
If $b_{0} \not\equiv b_{2}, \: b_{1} \not\equiv b_{3} \pmod{2}$, 
then $N_{4} \equiv N_{8} N_{16} - 4 (b_{0} b_{2} + b_{1} b_{3}) + 2 \pmod{16}$. 
\end{lem}
\begin{proof}
From Lemma~$\ref{lem:2.2}$~$(3)$ and $(4)$, 
we have 
\begin{align*}
N_{4} 
&= N_{8} N_{16} - 4 b_{0} b_{2} + 2 (b_{1}^{2} + b_{3}^{2}) \\ 
&= N_{8} N_{16} - 4 ( b_{0} b_{2} + b_{1} b_{3} ) + 2 ( b_{1} + b_{3} )^{2} \\ 
&\equiv N_{8} N_{16} - 4 ( b_{0} b_{2} + b_{1} b_{3} ) + 2 \hspace{2cm} \pmod{16}. 
\end{align*}
\end{proof}

\begin{lem}\label{lem:4.3}
If $c_{0} \not\equiv c_{2}, \: c_{1} \not\equiv c_{3} \pmod{2}$, 
then 
\begin{align*}
\Re{(\alpha_{2})} \equiv (- 1)^{c_{2}} + 2 (c_{0} c_{2} + c_{1} c_{3}) \pmod{8}, \\ 
\Im{(\alpha_{2})} \equiv (- 1)^{c_{1}} + 2 (c_{0} c_{2} + c_{1} c_{3}) \pmod{8}. 
\end{align*}
\end{lem}
\begin{proof}
From Lemma~$\ref{lem:2.3}$, 
we have 
\begin{align*}
\Re{(\alpha_{2})} 
&= c_{0}^{2} - c_{2}^{2} + 2 c_{1} c_{3} \\ 
&= ( c_{0} - c_{2} )^{2} - 2 c_{2}^{2} + 2 ( c_{0} c_{2} + c_{1} c_{3} ) \\ 
&\equiv 1 - 2 c_{2}^{2} + 2 ( c_{0} c_{2} + c_{1} c_{3} )  \\ 
&\equiv ( - 1 )^{c_{2}} + 2 ( c_{0} c_{2} + c_{1} c_{3} ) \hspace{2cm} \pmod{8}. 
\end{align*}
In the same way, 
the remaining one can also be proved. 
\end{proof}

\begin{rem}\label{rem:4.5}
For any $a, b, c, d \in \mathbb{Z}$, 
the following hold: 
\begin{enumerate}
\item[$(1)$] $a b + c d \equiv a c + b d \pmod{2}$ when $a + b \not\equiv c + d \pmod{2}$; 
\item[$(2)$] $4 a b (a^{2} +  b^{2}) \equiv 0 \pmod{8}$; 
\item[$(3)$] $a^{4} + b^{4} + 2 a^{2} b^{2} + 4 a b \equiv \dfrac{1}{2} \left\{ 1 - (- 1)^{a + b} \right\} \pmod{8}$. 
\end{enumerate}
\end{rem}

\begin{lem}\label{lem:4.4}
If $e_{0} + e_{4} \not\equiv e_{2} + e_{6}, \: e_{1} + e_{5} \not\equiv e_{3} + e_{7} \pmod{2}$, 
then 
$$
2 ( e_{0} e_{4} + e_{2} e_{6} + e_{1} e_{5} + e_{3} e_{7} ) \equiv b_{0} b_{2} + b_{1} b_{3} + c_{0} c_{2} + c_{1} c_{3} \pmod{4}. 
$$
\end{lem}
\begin{proof}
From Remark~$\ref{rem:4.5}$~$(1)$, we have  
\begin{align*}
b_{0} b_{2} + b_{1} b_{3} + c_{0} c_{2} + c_{1} c_{3} 
&= 2 (a_{0} + a_{8}) (a_{2} + a_{10}) + 2 (a_{4} + a_{12}) (a_{6} + a_{14}) \\ 
&\quad + 2 (a_{1} + a_{9}) (a_{3} + a_{11}) + 2 (a_{5} + a_{13}) (a_{7} + a_{15}) \\ 
&\equiv 2 (e_{0} e_{2} + e_{4} e_{6}) + 2 (e_{1} e_{3} + e_{5} e_{7}) && \\ 
&\equiv 2 (e_{0} e_{4} + e_{2} e_{6}) + 2 (e_{1} e_{5} + e_{3} e_{7}) && \pmod{4}. 
\end{align*}
\end{proof}

\begin{lem}\label{lem:4.6}
If $e_{0} + e_{4} \not\equiv e_{2} + e_{6}, \: e_{1} + e_{5} \not\equiv e_{3} + e_{7} \pmod{2}$, 
then 
\begin{align*}
\Re{(\alpha_{1})} \equiv (- 1)^{b_{2}} + 2 ( b_{0} b_{2} + b_{1} b_{3} + c_{0} c_{2} + c_{1} c_{3} ) \pmod{8}, \\ 
\Im{(\alpha_{1})} \equiv (- 1)^{b_{1}} + 2 ( b_{0} b_{2} + b_{1} b_{3} + c_{0} c_{2} + c_{1} c_{3})\pmod{8}. 
\end{align*}
\end{lem}
\begin{proof}
Noting that either $e_{1}^{2} - e_{5}^{2}$ or $e_{3}^{2} - e_{7}^{2}$ is a multiple of $4$, 
from Lemma~$\ref{lem:2.3}$,  
\begin{align*}
\Re{(\alpha_{1})} 
&\equiv 
e_{0}^{4} + e_{4}^{4} - e_{2}^{4} - e_{6}^{4} - 6 e_{0}^{2} e_{4}^{2} + 6 e_{2}^{2} e_{6}^{2} \\ 
&\quad + 4 (e_{2} e_{6} + e_{1} e_{7} + e_{3} e_{5}) (e_{0}^{2} - e_{4}^{2}) - 4 (e_{0} e_{6} + e_{2} e_{4} - e_{1} e_{5}) (e_{1}^{2} - e_{5}^{2}) \\ 
&\quad + 4 (e_{0} e_{4} + e_{1} e_{3} - e_{5} e_{7}) (e_{2}^{2} - e_{6}^{2}) - 4 (e_{0} e_{2} - e_{4} e_{6} + e_{3} e_{7}) (e_{3}^{2} - e_{7}^{2}) \pmod{8}. 
\end{align*}
Since $-2 \equiv 6, \: - 4 \equiv 4 \pmod{8}$, 
\begin{align*}
\Re{(\alpha_{1})} 
&\equiv 
e_{0}^{4} + e_{4}^{4} - e_{2}^{4} - e_{6}^{4} + 2 e_{0}^{2} e_{4}^{2} - 2 e_{2}^{2} e_{6}^{2} \\ 
&\quad + 4 (e_{2} e_{6} + e_{1} e_{7} + e_{3} e_{5}) (e_{0}^{2} + e_{4}^{2}) + 4 (e_{0} e_{6} + e_{2} e_{4} + e_{1} e_{5}) (e_{1}^{2} + e_{5}^{2}) \\ 
&\quad + 4 (e_{0} e_{4} + e_{1} e_{3} + e_{5} e_{7}) (e_{2}^{2} + e_{6}^{2}) + 4 (e_{0} e_{2} + e_{4} e_{6} + e_{3} e_{7}) (e_{3}^{2} + e_{7}^{2}) \pmod{8}. 
\end{align*}
From Remark~$\ref{rem:4.5}$~$(2)$, 
\begin{align*}
\Re{(\alpha_{1})} 
&\equiv 
e_{0}^{4} + e_{4}^{4} - e_{2}^{4} - e_{6}^{4} + 2 e_{0}^{2} e_{4}^{2} - 2 e_{2}^{2} e_{6}^{2} \\ 
&\quad + 4 (e_{2} e_{6} + e_{1} e_{7} + e_{3} e_{5}) (e_{0}^{2} + e_{4}^{2}) + 4 (e_{0} e_{6} + e_{2} e_{4} + e_{1} e_{5}) (e_{1}^{2} + e_{5}^{2}) \\ 
&\quad + 4 (e_{0} e_{4} + e_{1} e_{3} + e_{5} e_{7}) (e_{2}^{2} + e_{6}^{2}) + 4 (e_{0} e_{2} + e_{4} e_{6} + e_{3} e_{7}) (e_{3}^{2} + e_{7}^{2}) \\ 
&\quad + 4 e_{0} e_{4} (e_{0}^{2} + e_{4}^{2}) - 4 e_{1} e_{5} (e_{1}^{2} + e_{5}^{2}) + 4 e_{2} e_{6} (e_{2}^{2} + e_{6}^{2}) - 4 e_{3} e_{7} (e_{3}^{2} + e_{7}^{2}) && \\ 
&\equiv 
e_{0}^{4} + e_{4}^{4} - e_{2}^{4} - e_{6}^{4} + 2 e_{0}^{2} e_{4}^{2} - 2 e_{2}^{2} e_{6}^{2} \\ 
&\quad + 4 (e_{0} e_{4} + e_{2} e_{6} + e_{1} e_{7} + e_{3} e_{5}) (e_{0}^{2} + e_{4}^{2}) + 4 (e_{0} e_{6} + e_{2} e_{4}) (e_{1}^{2} + e_{5}^{2}) \\ 
&\quad + 4 (e_{0} e_{4} + e_{2} e_{6} + e_{1} e_{3} + e_{5} e_{7}) (e_{2}^{2} + e_{6}^{2}) + 4 (e_{0} e_{2} + e_{4} e_{6}) (e_{3}^{2} + e_{7}^{2}) &&\pmod{8}. 
\end{align*}
From Remark~$\ref{rem:4.5}$~$(1)$, 
\begin{align*}
\Re{(\alpha_{1})} 
&\equiv 
e_{0}^{4} + e_{4}^{4} - e_{2}^{4} - e_{6}^{4} + 2 e_{0}^{2} e_{4}^{2} - 2 e_{2}^{2} e_{6}^{2} \\ 
&\quad + 4 (e_{0} e_{4} + e_{2} e_{6} + e_{1} e_{5} + e_{3} e_{7}) (e_{0}^{2} + e_{4}^{2}) + 4 (e_{0} e_{4} + e_{2} e_{6}) (e_{1}^{2} + e_{5}^{2}) \\ 
&\quad + 4 (e_{0} e_{4} + e_{2} e_{6} + e_{1} e_{5} + e_{3} e_{7}) (e_{2}^{2} + e_{6}^{2}) + 4 (e_{0} e_{4} + e_{2} e_{6}) (e_{3}^{2} + e_{7}^{2}) && \\ 
&\equiv e_{0}^{4} + e_{4}^{4} + 2 e_{0}^{2} e_{4}^{2} - ( e_{2}^{4} + e_{6}^{4} + 2 e_{2}^{2} e_{6}^{2} ) \\ 
&\quad + 4 (e_{0} e_{4} + e_{2} e_{6} + e_{1} e_{5} + e_{3} e_{7}) (e_{0}^{2} + e_{4}^{2} + e_{2}^{2} + e_{6}^{2}) \\ 
&\quad + 4 (e_{0} e_{4} + e_{2} e_{6}) (e_{1}^{2} + e_{5}^{2} + e_{3}^{2} + e_{7}^{2}) &&\pmod{8}. 
\end{align*}
 Since $e_{0}^{2} + e_{2}^{2} + e_{4}^{2} + e_{6}^{2}$ and $e_{1}^{2} + e_{3}^{2} + e_{5}^{2} + e_{7}^{2}$ are odd numbers, 
\begin{align*}
\Re{(\alpha_{1})} 
&\equiv 
( e_{0}^{4} + e_{4}^{4} + 2 e_{0}^{2} e_{4}^{2} + 4 e_{0} e_{4} ) - ( e_{2}^{4} + e_{6}^{4} + 2 e_{2}^{2} e_{6}^{2} + 4 e_{2} e_{6} ) \\ 
&\quad + 4 (e_{0} e_{4} + e_{2} e_{6} + e_{1} e_{5} + e_{3} e_{7}) && \pmod{8}. 
\end{align*}
From Remark~$\ref{rem:4.5}$~$(3)$ and Lemma~$\ref{lem:4.4}$, 
\begin{align*}
\Re{(\alpha_{1})} 
&\equiv \frac{1}{2} \left\{ 1 - (- 1)^{e_{0} + e_{4}} - 1 + (- 1)^{e_{2} + e_{6}} \right\} + 4 ( e_{0} e_{4} + e_{2} e_{6} + e_{1} e_{5} + e_{3} e_{7} ) && \\ 
&\equiv (- 1)^{b_{2}} + 4 ( e_{0} e_{4} + e_{2} e_{6} + e_{1} e_{5} + e_{3} e_{7} ) && \\ 
&\equiv (- 1)^{b_{2}} + 2 (b_{0} b_{2} + b_{1} b_{3} + c_{0} c_{2} + c_{1} c_{3}) && \pmod{8}. 
\end{align*}
In the same way, the remaining one can also be proved. 
\end{proof}

\begin{lem}\label{lem:4.7}
Suppose that 
$\Re{(\alpha_{1})}$, $\Im{(\alpha_{1})}$, $\Re{(\alpha_{2})}$, $\Im{(\alpha_{2})} \in \{ 8 m \pm 1 \: | \: m \in \mathbb{Z} \}$. 
Then $N_{4} \equiv N_{8} N_{16} + 2 \pmod{16}$ holds. 
\end{lem}
\begin{proof}
Under the assumption, 
since $c_{0} \not\equiv c_{2}, \: c_{1} \not\equiv c_{3} \pmod{2}$, 
we have 
$$
c_{0} c_{2} + c_{1} c_{3} \equiv 0, \quad b_{0} b_{2} + b_{1} b_{3} + c_{0} c_{2} + c_{1} c_{3} \equiv 0 \pmod{4} 
$$
from Lemmas~$\ref{lem:4.3}$ and $\ref{lem:4.6}$, respectively. 
That is, $b_{0} b_{2} + b_{1} b_{3} \equiv 0 \pmod{4}$. 
From this and Lemma~$\ref{lem:4.2}$, 
the lemma is proved. 
\end{proof}

\begin{lem}\label{lem:4.8}
Let $p_{i} = a_{i}^{2} + b_{i}^{2} \equiv 1 \pmod{8}$ be a prime 
with $a_{i} \pm b_{i} \in \{8 m \pm 1 \mid m \in \mathbb{Z} \}$ for each $1 \leq i \leq r$,  
let $p_{r + 1}, \ldots, p_{r + s} \equiv 7 \pmod{8}$ be primes, 
and let $k_{1}, \ldots, k_{r + s}$ be non-negative integers. 
If $\alpha \in \mathbb{Z} [\sqrt{-1}]$ satisfies 
$\alpha \overline{\alpha} = 2 p_{1}^{k_{1}} \cdots p_{r}^{k_{r}} p_{r + 1}^{2 k_{r + 1}} \cdots p_{r + s}^{2 k_{r + s}}$, then 
\begin{align*}
\Re{(\alpha)}, \: \Im{(\alpha)} \in \{8 m \pm 1 \mid m \in \mathbb{Z} \}. 
\end{align*}
\end{lem}
\begin{proof}
Obviously, the statement of the lemma holds when $k_{1} = k_{2} = \cdots = k_{r + s} = 0$. 
Let $\beta \in \mathbb{Z} [\sqrt{-1}]$ satisfy 
\begin{align*}
\Re{(\beta)}, \: \Im{(\beta)} \in \{8 m \pm 1 \mid m \in \mathbb{Z} \}. 
\end{align*}
For $1 \leq i \leq r$, let $\pi_{i} := a_{i} + \sqrt{- 1} b_{i}$. Then the following holds: 
\begin{align*}
\Re{(\beta \pi_{i})} &= \Re{(\beta)} \Re{(\pi_{i})} - \Im{(\beta)} \Im{(\pi_{i})} \equiv \pm 1 \pmod{8}, \\ 
\Im{(\beta \pi_{i})} &= \Re{(\beta)} \Im{(\pi_{i})} + \Im{(\beta)} \Re{(\pi_{i})} \equiv \pm 1 \pmod{8}. 
\end{align*}
Also, for $r + 1 \leq i \leq r + s$, the following holds: 
\begin{align*}
\Re{(\beta p_{i})} &= \Re{(\beta)} p_{i} \equiv \pm 1 \pmod{8}, \\ 
\Im{(\beta p_{i})} &= \Im{(\beta)} p_{i} \equiv \pm 1 \pmod{8}. 
\end{align*}
Therefore, by induction, we can prove that the statement of the lemma holds for arbitrary non-negative integers $k_{1}, k_{2}, \ldots, k_{r + s}$. 
\end{proof}

\begin{proof}[Proof of Theorem~$\ref{thm:4.1}$]
We prove by contradiction. 
Assume that there exist $a_{0}, \ldots, a_{15} \in \mathbb{Z}$ satisfying 
$$
D(a_{0}, a_{1}, \ldots, a_{15}) 
= 64 p_{1}^{k_{1}} \cdots p_{r}^{k_{r}} p_{r + 1}^{k_{r + 1}} \cdots p_{r + s}^{k_{r + s}} Q, 
$$
where $Q = 1$ or $q_{1} q_{2} \cdots q_{t}$. 
Then, it follows from Theorem~$\ref{thm:3.1}$ and Lemma~$\ref{lem:2.5}$ that there exist $l_{i}, m_{i}, n_{i} \geq 0$ such that 
\begin{align*}
N_{1} &= 2 p_{1}^{l_{1}} \cdots p_{r}^{l_{r}} p_{r + 1}^{2 l_{r + 1}} \cdots p_{r + s}^{2 l_{r + s}}, \\
N_{2} &= 2 p_{1}^{m_{1}} \cdots p_{r}^{m_{r}} p_{r + 1}^{2 m_{r + 1}} \cdots p_{r + s}^{2 m_{r + s}}, \\
N_{4} &= 2 p_{1}^{n_{1}} \cdots p_{r}^{n_{r}} p_{r + 1}^{2 n_{r + 1}} \cdots p_{r + s}^{2 n_{r + s}}, \\ 
N_{8} N_{16} &= 8 p_{1}^{k_{1} - l_{1} - m_{1} - n_{1}} \cdots p_{r}^{k_{r} - l_{r} - m_{r} - n_{r}} 
p_{r + 1}^{k_{r + 1} - 2 (l_{r + 1} + m_{r + 1} + n_{r + 1})} \cdots p_{r + s}^{k_{r + s} - 2 (l_{r + s} + m_{r + s} + n_{r + s})} Q. 
\end{align*}
Note that $N_{4} \equiv 2$, $N_{8} N_{16} \equiv 8 \pmod{16}$. 
On the other hand, 
applying Lemma~$\ref{lem:4.8}$ to $N_{1}$ and $N_{2}$, 
using Lemma~$\ref{lem:4.7}$, we have $N_{4} \equiv N_{8} N_{16} + 2 \pmod{16}$. 
This is a contradiction. 
\end{proof}

\section{Possible values}

In this section, 
we determine all possible values. 
From the Kaiblinger's result mentioned in the introduction, 
we know that $\mathbb{Z}_{{\rm odd}} \cup 128 \mathbb{Z} \subset S(16)$. 
Since $S(n)$ is closed under multiplication, 
to prove the following theorem completes the proof of Theorem~$\ref{thm:1.1}$. 

\begin{thm}\label{thm:5.1}
The following hold: 
\begin{enumerate}
\item[$(1)$] Suppose that $p$ is a prime with $p \equiv 5 \pmod{8}$, then $64 p \in S(16)$; 
\item[$(2)$] Suppose that $p$ is a prime with $p \equiv 3 \pmod{8}$, then $64 p^{2} \in S(16)$; 
\item[$(3)$] Suppose that $p$ is a prime with $p = a^{2} + b^{2} \equiv 1 \pmod{8}$ and $a \pm b \in \left\{ 8 m \pm 3 \mid m \in \mathbb{Z} \right\}$, 
then $64 p \in S(16)$. 
\end{enumerate}
\end{thm}

To prove Theorem~$\ref{thm:5.1}$, we use the following lemma. 

\begin{lem}\label{lem:5.2}
For any $k$, $l$, $m$, $n \in \mathbb{Z}$, the following hold: 
\begin{flalign*}
&(1)&&32 \left\{ (8 k + 3)^{2} + (8 l + 1)^{2} \right\} \in S(16); \\
&(2)&&64 \left\{ (4 k - 1)^{2} + 2 (4 l - 1)^{2} \right\}^{2} \in S(16); \\
&(3)&&32 \left\{ (4 k - 1)^{2} - (4 m - 2)^{2} + 2 (2 l - 1) 4 n \right\}^{2} \\ 
& &&+ 32 \left\{ (2 l - 1)^{2} - (4 n)^{2} - 2 (4 k - 1) (4 m - 2) \right\}^{2} \in S(16). \hspace{6cm}
&
\end{flalign*}
\end{lem}
 
\begin{proof}
First, we prove (1). Let 
\begin{gather*}
a_{0} = k, \: \: 
a_{1} = l, \: \: 
a_{2} = - k, \: \: 
a_{3} = - l, \: \: 
a_{4} = k, \: \: 
a_{5} = l, \: \: 
a_{6} = - k - 1, \: \: 
a_{7} = - l, \\ 
a_{8} = k, \: \: 
a_{9} = l, \: \: 
a_{10} = - k - 1, \: \: 
a_{11} = - l, \: \: 
a_{12} = k, \: \: 
a_{13} = l, \: \: 
a_{14} = - k - 1, \: \: 
a_{15} = - l - 1. 
\end{gather*}
Then 
$D(a_{0}, a_{1}, \ldots, a_{15}) = 32 \left\{ (8 k + 3)^{2} + (8 l + 1)^{2} \right\}$. 
Next, we prove (2). 
Let 
\begin{gather*}
a_{0} = k + l, \: \: 
a_{1} = k - l, \: \: 
a_{2} = 1 - l, \: \: 
a_{3} = 1 - l, \: \: 
a_{4} = 1 - k - l, \: \: 
a_{5} = l - k, \: \: 
a_{6} = l, \: \: 
a_{7} = l, \\ 
a_{8} = k + l, \: \: 
a_{9} = k - l, \: \: 
a_{10} = - l, \: \: 
a_{11} = - l, \: \: 
a_{12} = 1 - k - l, \: \: 
a_{13} = l - k, \: \: 
a_{14} = l, \: \: 
a_{15} = l. 
\end{gather*}
Then 
$D(a_{0}, a_{1}, \ldots, a_{15}) = 64 \left\{ (4 k - 1)^{2} + 2 (4 l - 1)^{2} \right\}^{2}$. 
Finally, we prove (3). 
Let 
\begin{gather*}
a_{0} = k, \: \: 
a_{1} = l', \: \: 
a_{2} = m, \: \: 
a_{3} = n, \: \: 
a_{4} = - k, \: \: 
a_{5} = - l', \: \: 
a_{6} = 1 - m, \: \: 
a_{7} = - n, \\ 
a_{8} = k, \: \: 
a_{9} = l', \: \: 
a_{10} = m, \: \: 
a_{11} = n, \: \: 
a_{12} = 1 - k, \: \: 
a_{13} = (-1)^{l} - l', \: \: 
a_{14} = 1 - m, \: \: 
a_{15} = - n, 
\end{gather*}
where $l' := l / 2$ when $l$ is even; $l' := (l - 1) / 2$ when $l$ is odd. 
Then 
$D(a_{0}, a_{1}, \ldots, a_{15})$ becomes
\begin{align*}
32 \left\{ (4 k - 1)^{2} - (4 m - 2)^{2} + 2 (2 l - 1) 4 n \right\}^{2} + 32 \left\{ (2 l - 1)^{2} - (4 n)^{2} - 2 (4 k - 1) (4 m - 2) \right\}^{2}. 
\end{align*}
\end{proof}
 
\begin{proof}[Proof of Theorem~$\ref{thm:5.1}$]
First, we prove (1). 
Suppose that $p$ is a prime with $p \equiv 5 \pmod{8}$, 
then there exist $r, s \in \mathbb{Z}$ such that 
$$
p = (4 r + 2)^{2} + (4 s + 1)^{2}. 
$$
Here, we can take $r$ and $s$ so that $r \equiv s \pmod{2}$ holds: if $r \not\equiv s \pmod{2}$, then 
\begin{align*}
p = (4(- r' - 1) + 2)^{2} + (4 s + 1)^{2} = (4 r' + 2)^{2} + (4 s + 1)^{2}, 
\end{align*}
where $r' := - r - 1$, and $r' \equiv s \pmod{2}$ holds. 
Put $k := (r + s) / 2$ and $l := (r - s) / 2$. 
Then 
\begin{align*}
2 p 
&= 2 \left\{ (4 r + 2)^{2} + (4 s + 1)^{2} \right\} \\ 
&= \left\{ (4 r + 2) + (4 s + 1) \right\}^{2} + \left\{ (4 r + 2) - (4 s + 1) \right\}^{2} \\ 
&= \left\{ 4 (r + s) + 3 \right\}^{2} + \left\{ 4 (r - s) + 1 \right\}^{2} \\ 
&= (8 k + 3)^{2} + (8 l + 1)^{2}. 
\end{align*}
Therefore, from Lemma~$\ref{lem:5.2}$~$(1)$, we have 
$$
64 p = 32 \left\{ (8 k + 3)^{2} + (8 l + 1)^{2} \right\} \in S(16). 
$$
Next, we prove (2). 
Suppose that $p$ is a prime with $p \equiv 3 \pmod{8}$, 
then $p$ is expressible in the form $a^{2} + 2 b^{2}$, where $a$ and $b$ are odd numbers. 
Thus, there exist $k, l \in \mathbb{Z}$ satisfying 
$$
p = (4 k - 1)^{2} + 2 (4 l - 1)^{2}. 
$$
Therefore, from Lemma~$\ref{lem:5.2}$~$(2)$, we have
$$
64 p^{2} = 64 \left\{ (4 k - 1)^{2} + 2 (4 l - 1)^{2} \right\}^{2} \in S(16). 
$$
Finally, we prove (3). 
For any $v := (r, s, t, u) \in \mathbb{Z}^{4}$, 
let 
$$
h_{v}(x) := r + s x + t x^{2} + u x^{3} \in \mathbb{Z}[x]. 
$$
Suppose that $p$ is a prime with $p = a^{2} + b^{2} \equiv 1 \pmod{8}$ and $a \pm b \in \left\{ 8 m \pm 3 \mid m \in \mathbb{Z} \right\}$, 
then there exists $v = (r, s, t, u) \in \mathbb{Z}^{4}$ with $r + t \not\equiv s + u \pmod{2}$ satisfying
$$
p = h_{v}(\zeta_{8}) h_{v}(\zeta_{8}^{3}) h_{v}(\zeta_{8}^{5}) h_{v}(\zeta_{8}^{7}) = (r^{2} - t^{2} + 2 s u)^{2} + (s^{2} - u^{2} - 2 r t)^{2}. 
$$ 
Here, we can assume that 
$$
r \not\equiv s \equiv t \equiv u \pmod{2}. 
$$
If this condition does not hold, 
then we replace $v$ by 
\begin{align*}
v := 
\begin{cases}
(s, t, u, - r) & \text{if} \:\: s \not\equiv t \equiv u \equiv r \pmod{2}, \\ 
(t, u, - r, - s) & \text{if} \:\: t \not\equiv u \equiv r \equiv s \pmod{2}, \\ 
(u, r, - s, t) & \text{if} \:\: u \not\equiv r \equiv s \equiv t \pmod{2}. 
\end{cases}
\end{align*}
Noting that 
$$
( r^{2} - t^{2} + 2 s u ) + ( s^{2} - u^{2} - 2 r t ) \in \left\{ \pm (a + b), \pm (a - b) \right\} \subset \left\{ 8 m \pm 3 \: \mid \: m \in \mathbb{Z} \right\}, 
$$
we have 
\begin{align*}
\pm 3 
&\equiv 
(r^{2} - t^{2} + 2 s u) + (s^{2} - u^{2} - 2 r t) && \\ 
&\equiv 2 (r^{2} + s^{2}) - (r + t)^{2} - (s - u)^{2} && \\ 
&\equiv 1 - (s - u)^{2} &&\pmod{8}. 
\end{align*}
Thus $s \not\equiv u \pmod{4}$. 
We can assume further that 
\begin{align*}
(t + s, t + u) \equiv (2, 0) \pmod{4}. 
\end{align*}
If this condition does not hold, 
then we replace $v$ by $v := (- r, u, t, s)$. 
Let 
$$
c_{0} := r - u, \quad c_{1} := r + s, \quad c_{2} := t + s, \quad c_{3} := t + u. 
$$
Then, since 
$$
c_{0} \equiv c_{1} \equiv 1 \pmod{2}, \quad c_{2} \equiv 2, \: c_{3} \equiv 0 \pmod{4}, 
$$
there exist $k, l, m, n \in \mathbb{Z}$ such that 
\begin{align*}
(c_{0}, c_{1}, c_{2}, c_{3}) = 
\begin{cases}
(- 4 k + 1, 2 l - 1, - 4 m + 2, 4 n) & \text{if} \: \: c_{0} \equiv 1 \pmod{4}, \\ 
(4 k - 1, 2 l - 1, 4 m - 2, 4 n) & \text{if} \: \: c_{0} \equiv 3 \pmod{4}. 
 \end{cases}
\end{align*}
Since $(1 + \zeta_{8}^{k}) h_{v}(\zeta_{8}^{k}) = c_{0} + c_{1} \zeta_{8}^{k} + c_{2} \zeta_{8}^{2 k} + c_{3} \zeta_{8}^{3 k}$, 
we have 
\begin{align*}
2 p 
&= (1 + \zeta_{8}) (1 + \zeta_{8}^{3}) (1 + \zeta_{8}^{5}) (1 + \zeta_{8}^{7}) h_{v}(\zeta_{8}) h_{v}(\zeta_{8}^{3}) h_{v}(\zeta_{8}^{5}) h_{v}(\zeta_{8}^{7}) \\ 
&= \prod_{k \in \{ 1, 3, 5, 7 \}} ( c_{0} + c_{1} \zeta_{8}^{k} + c_{2} \zeta_{8}^{2 k} + c_{3} \zeta_{8}^{3 k}) \\ 
&= (c_{0}^{2} - c_{2}^{2} + 2 c_{1} c_{3})^{2} + (c_{1}^{2} - c_{3}^{2} - 2 c_{0} c_{2})^{2} \\ 
&= \left\{ (4 k - 1)^{2} - (4 m - 2)^{2} + 2 (2 l - 1) 4 n \right\}^{2} + \left\{ (2 l - 1)^{2} - (4 n)^{2} - 2 (4 k - 1) (4 m - 2) \right\}^{2}. 
\end{align*}
Therefore, from Lemma~$\ref{lem:5.2}$~$(3)$, the proof of Theorem~$\ref{thm:5.1}$~$(3)$ is complete. 
\end{proof}




\clearpage

\bibliography{reference}
\bibliographystyle{plain}

\medskip
\begin{flushleft}
Yuka Yamaguchi\\
Faculty of Education \\ 
University of Miyazaki \\
1-1 Gakuen Kibanadai-nishi\\ 
Miyazaki 889-2192 \\
JAPAN\\
y-yamaguchi@cc.miyazaki-u.ac.jp
\end{flushleft}

\medskip
\begin{flushleft}
Naoya Yamaguchi\\
Faculty of Education \\ 
University of Miyazaki \\
1-1 Gakuen Kibanadai-nishi\\ 
Miyazaki 889-2192 \\
JAPAN\\
n-yamaguchi@cc.miyazaki-u.ac.jp
\end{flushleft}

\end{document}